\documentclass[11pt]{article}

\usepackage[a4paper,margin=30mm]{geometry}
\usepackage[T1]{fontenc}
\usepackage{lmodern}
\usepackage{microtype}
\usepackage{amsmath,amssymb,amsthm,mathtools}
\usepackage[hidelinks]{hyperref}
\usepackage[nameinlink]{cleveref}

\newtheorem{theorem}{Theorem}[section]
\newtheorem{lemma}[theorem]{Lemma}

\newtheorem{corollary}[theorem]{Corollary}

\crefname{theorem}{theorem}{theorems}
\crefname{lemma}{lemma}{lemmas}
\crefname{proposition}{proposition}{propositions}
\crefname{corollary}{corollary}{corollaries}
\crefname{remark}{remark}{remarks}
\crefname{conjecture}{conjecture}{conjectures}

\newcommand{\I}{\mathcal I}

\newcommand{\Prob}{\mathbb P}
\newcommand{\1}{\mathbf 1}
\newcommand{\tr}{\operatorname{tr}}

\title{A non-trivial bound for 3AP-intersecting families}

\author{Peter Keevash\thanks{Mathematical Institute, University of Oxford, UK. Supported by ERC Advanced Grant 883810.}}

\date{}

\begin{document}
\maketitle

\begin{abstract}
A family $F$ of subsets of $[n]$ is \emph{3AP-intersecting} if every two
members have intersection containing a non-trivial three-term arithmetic
progression. We prove that there is an absolute constant $c>0$ such that
any such $F$ has size at most $(\tfrac12 - c)2^n$. This is the first non-trivial
progress towards a conjecture of Simonovits and S\'os that 
the maximum possible size is $2^{n-3}$.
More generally, we show the same bound for $H$-intersecting families
whenever $H$ is a $3$-graph on $[n]$ with bounded codegrees.
A clique shows that this is sharp, in that the bounded codegree assumption cannot be removed.
\end{abstract}

\section{Introduction}

Given a hypergraph $H$ and a family $F$ of subsets of a common vertex set $V$,
we say $F$ is \emph{\(H\)-intersecting} if \(A\cap B\) contains an edge of $H$ for every \(A,B\in F\). 
We say that $F$ on $[n]=\{1,\dots,n\}$ is 3AP-intersecting if it is $H$-intersecting
with $H$ being the $3$-graph on $[n]$ whose edges are all non-trivial $3$-term arithmetic progressions (3APs).
A conjecture of Simonovits and S\'os (personal communication to the authors of~\cite{CGFS1986})
states that any 3AP-intersecting family has size at most $2^{n-3}$. This would clearly be sharp,
as shown by the family of all sets containing some fixed 3AP. They also made the analogous conjecture
for triangle-intersecting families of graphs, which can be viewed as an $H$-intersecting problem
in the auxiliary $3$-graph on $V=E(K_n)$ whose edges are all triples $\{ab,bc,ca\}$.
This conjecture was proved using spectral methods by Ellis, Filmus and Friedgut~\cite{EFF2012}. 
However, it seems that their methods cannot work for the 3AP problem.

There is a rich field of related theorems and conjectures on intersection problems, see the survey \cite{E2022}. 
A long-standing conjecture of Alon stated that there is an absolute constant $c>0$ such that 
for any graph $H$ that is not a star-forest, any $H$-intersecting family of graphs on $[n]$
contains at most $\tfrac12 - c$ proportion of all graphs. This was recently proved by Han and Wang~\cite{HanWang2026},
building on a new concentration of measure inequality due to Gillott~\cite{Gillott2026}.
In this note we show that a similar method gives the same type of bound
for the Simonovits-S\'os 3AP-intersection conjecture. More generally,
the argument applies to any $3$-graph $H$ with bounded codegrees.
We write $\Delta_2(H) := \max_{\{x,y\}\subseteq V(H)}  |\{e\in E(H):\{x,y\}\subseteq e\}|\}$,

\begin{theorem}\label{thm:main}
For every $\Delta \in \mathbb{N}$ there is $c_\Delta>0$
such that for any $3$-graph $H$ with $\Delta_2(H)\le\Delta$
and $H$-intersecting family $F$ of subsets of $V(H)$
we have $|F| \le (\tfrac12 - c_\Delta)2^{|V(H)|}$.
\end{theorem}

\begin{corollary}\label{cor:3ap}
There is an absolute constant \(c>0\) such that every
3AP-intersecting family on $[n]$ has size at most $(\tfrac12 - c)2^n$. 
\end{corollary}

To see sharpness of Theorem \ref{thm:main} suppose that $H$ is obtained from a complete $3$-graph on $m$ vertices by adding some isolated vertices.
Let $F$ consist of all subsets that intersect the clique in at least $(m+3)/2$ vertices.
Then $F$ has density $1/2 - o_m(1)$ as $m \to \infty$. If a symmetric example is desired
one could let $H$ be a disjoint union of cliques.

\section{Proof strategy}

Let $F$ be an $H$-intersecting family on $V(H)$.
We identify subsets of $V(H)$ with indicator vectors in $\mathbb{F}_2^{V(H)}$.
We write \(\1\) for the all-ones vector and
let $\I(H)$ be the family of all independent sets in $H$.
Then the sumsets $F+\I(H)$ and $F+\{\1\} = \{ V(H) \setminus A: A \in F\}$ are disjoint,
as if we had $A+I=B+\1$ for some $A,B$ in $F$ then $A \cap B \subseteq I$,
which contradicts $F$ being $H$-intersecting and $I$ being independent.
This observation reduces the proof to an isoperimetric statement 
in the Cayley graph on $\mathbb{F}_2^{V(H)}$ generated by $\I(H)$.
It suffices to show that if $F$ has density $p$ close to $1/2$ then it expands
to a set $F+\I(H)$ of density more than $1-p$. This is forced to intersect $F+\1$
simply by size considerations, so such $F$ cannot be $H$-intersecting.

The above viewpoint was well-known to experts, but did not lead to any progress
until recently due to the lack of a suitable isoperimetric statement. The key new
ingredient is the following concentration inequality of Gillott, which he applied to
obtain several new bounds, including the version of Corollary \ref{cor:3ap}
with 3AP replaced by 4AP. The following is a slight adjustment of
\cite[Theorem~3.1]{Gillott2026} that follows from the same proof.
For convenient comparison we retain the nonstandard notation used there,
writing $Y_P$ for the binomial random subset of $[m]$ where
each $\Prob(i \in Y_P)=p_i$ independently.

\begin{lemma}[Gillott]\label{lem:gillott-param}
For any sufficiently small $K>0$ there is  \(\eta=\eta(K)>0\)
such that if \(F\subseteq\mathbb{F}_2^m\) with $|F| \ge (\tfrac12-\eta)2^m$
then for at least $ (\tfrac12+\eta)2^m$ points  \(x\in\mathbb{F}_2^m\) 
there exists \(P=(p_i)\in[0,1]^m\) such that $\sum_i p_i^2\le K$
and $ \Prob(x+Y_P\in F)>K^{3/2}$.
\end{lemma}

To deduce that a family $F$ of density $\tfrac12-\eta$ for small enough $\eta>0$
cannot be 4AP-intersecting, one applies this lemma and size considerations to see
that there is $x' \in F$ such that $x=x'+\1$ has $P$ as in the lemma.
By Cauchy-Schwarz (see \cite[Theorem~3.3]{Gillott2026}),
the probability that $Y_P$ contains a 4AP is at most $K^2 < K^{3/2}$,
so there is nonzero probability that $x''=x+Y_P \in F$ and $Y_P$ contains no 4AP.
But then $x=x'+\1=x''+Y_P$ belongs to $F+\1$ and $F+\I(H)$, so $F$ cannot be $H$-intersecting.

This argument does not handle 3AP-intersecting families or Alon's graph intersection conjecture.
To resolve the latter, Han and Wang \cite{HanWang2026} employ an additional trick 
based on  Pl\"unnecke's inequality, which allows one to replace $\I(H)$ by the sumset $\I(H)+\I(H)$;
the following is a slight adaptation of their argument. For convenient notation we write $\mu$ for uniform measure.

\begin{lemma}\label{lem:additive-reduction}
Suppose that for some \(0<\eta<1/4\), every
\( A\subseteq\mathbb{F}_2^{V(H)}\) with \( \mu(A)\ge1/2-\eta\) 
satisfies \( \mu( A+\I(H)+\I(H))\ge\frac12+\eta\).
Then every \(H\)-intersecting family has density
at most $\tfrac12-\tfrac{\eta}{9}$.
\end{lemma}

\begin{proof}
Let \(F\) be a nonempty \(H\)-intersecting family of density \(\alpha\). 
As noted above, we have \((F+\I(H))\cap(F+\1)=\varnothing\).
As \(\varnothing\in\I(H)\) we have $\alpha \le 1/2$
and $\frac{|F+\I(H)|}{|F|} \le\frac{1-\alpha}{\alpha}$.
By Pl\"unnecke's inequality, $F$ has a nonempty subset $X$
with \(  |X+\I(H)+\I(H)| \le
  \left(\frac{1-\alpha}{\alpha}\right)^2|X| \).
By \cite[Lemma 2.2]{HanWang2026}, which can be applied
as $\I(H)+\I(H)$ generates $\mathbb{F}_2^{V(H)}$,
we deduce \(|X+\I(H)+\I(H)|\ge(1+2\eta)|X|\).
We deduce $\alpha \le \tfrac12-\tfrac{\eta}{9}$.
\end{proof}

Our final ingredient is an estimate for $\Prob\bigl(Y_P\notin\I(H)+\I(H)\bigr)$,
which plays the same role in our proof as the estimate of the probability 
that $Y_P$ contains a 4AP plays in Gillott's proof. For fixed $\Delta$ our
bound is $O_\Delta(K^2) < K^{3/2}$ for small $K$.

\begin{lemma}\label{lem:splitting}
Let \(H\) be a $3$-graph with \(\Delta_2(H)\le\Delta\).
Write $\I(H)$ for the family of independent sets in $H$.
Let \(Y_P\subseteq V(H)\) retain the vertices independently 
with probabilities \((p_v)\) and write \(K=\sum_{v\in V(H)}p_v^2\).
If \(\Delta K<1\) then
\(
  \Prob\bigl(Y_P\notin\I(H)+\I(H)\bigr)
  \le \frac{(\Delta K)^2}{1-\Delta K}.
\)
\end{lemma}

We prove this lemma in the next section. First we complete
the deduction of Theorem \ref{thm:main}, which is as above.

\begin{proof}[Proof of Theorem \ref{thm:main}]
Fix $\Delta$ and choose $0<K<1/\Delta$ 
such that $\frac{(\Delta K)^2}{1-\Delta K} < K^{3/2}$.
Let $\eta$ be as in Lemma \ref{lem:gillott-param}.
Let $H$ be a $3$-graph with  \(\Delta_2(H)\le\Delta\).
Let $F \subseteq\mathbb{F}_2^{V(H)}\) with \(\mu(F)\ge1/2-\eta\).
Let $A$ be the set of $x$ such that $P$ as in  Lemma \ref{lem:gillott-param} exists.
Then $\mu(A) \ge 1/2 + \eta$. By Lemma \ref{lem:splitting} 
we have $\Prob\bigl(Y_P\notin\I(H)+\I(H)\bigr) < K^{3/2}$.
Thus the events \(\{x+Y_P\in F\}\) and \(\{Y_P\in \I(H)+\I(H)\}\) intersect,
so $x \in F+\I(H)+\I(H)$. We deduce that $\mu(F+\I(H)+\I(H)) \ge \mu(A)  \ge 1/2 + \eta$. 
The theorem now follows by Lemma \ref{lem:additive-reduction}.
\end{proof}

\section{Proof of the lemma}\label{sec:noise}

It remains to prove Lemma \ref{lem:splitting}. Our first observation is that it suffices
to bound the probability that \(H[Y_P]\) contains a cycle in its incidence graph.
Indeed, if $H[Y_P]$ is a forest then it is $2$-colourable (by a greedy algorithm),
so $Y_P \in\I(H)+\I(H)$. 

To estimate this probability, we consider a shortest cycle in $H[Y_P]$
and write it as $x_1,e_1,x_2,e_2,\ldots,x_\ell,e_\ell,x_1$, where each
\(e_i=\{x_i,z_i,x_{i+1}\}\), reading indices modulo \(\ell\). 
Minimality implies that the $x_i$ and $z_i$ are all distinct. 

We bound the probability of seeing such a cycle by 
$\sum_{\ell \ge 2} \tr M^\ell$, where $M$ is the symmetric matrix with
$M_{xx}=0$ and $M_{xy} = \sqrt{p_xp_y}\sum_{z: xyz \in E(H)}p_z$.

Let \(\lambda_j\) be the eigenvalues of \(M\). Then
$\tr M^\ell = \sum_j \lambda_j^\ell \le\left(\sum_j\lambda_j^2\right)^{\ell/2}$,
where $\sum_j\lambda_j^2 = \tr M^2 = \sum_{x\ne y}p_xp_y
    \left(\sum_{z: xyz \in E(H) }p_z\right)^2 \le   
    2\Delta \sum_z p_z^2 \sum_{xy: xyz \in E(H)} p_xp_y
    \le \Delta^2K^2 $
by the codegree bounds and Cauchy-Schwarz.
The last inequality uses $p_x p_y \le (p_x^2+p_y^2)/2$
and the codegree bound. The lemma follows.

\subsection*{Statement on AI use}

This proof was found by GPT-6 Astra following a hint by the author.
The author simplified and rewrote the proof.

\end{document}